\documentclass[12pt]{amsart}
\usepackage{graphicx}
\usepackage{amsmath, amsxtra, amssymb, latexsym}
\usepackage{hyperref}

\newtheorem{thm}{Theorem}[section]

\newtheorem{lem}[thm]{Lemma}

\theoremstyle{definition}

\theoremstyle{remark}

\numberwithin{equation}{section}

\def\XXint#1#2#3{{\setbox0=\hbox{$#1{#2#3}{\int}$}
  \vcenter{\hbox{$#2#3$}}\kern-.5\wd0}}

\begin{document}

\title[Smooth metric measure spaces]
{$L^2_f$ harmonic 1-forms on smooth metric measure spaces with positive $\lambda_1(\Delta_f)$}

\author{Jiuru Zhou}
\address{School of Mathematical Science \\
Yangzhou University \\
Yangzhou,  Jiangsu 225002, China}
\email{zhoujiuru@yzu.edu.cn}




\date{\today}

\begin{abstract}
 In this paper, we study vanishing and splitting results on a complete smooth metric measure space $(M^n,g,\mathrm{e}^{-f}\mathrm{d}v)$ with various negative $m$-Bakry-\'Emery Ricci curvature lower bounds in terms of the first spectrum $\lambda_1(\Delta_f)$ of the weighted Laplacian $\Delta_f$, i.e. $\mathrm{Ric}_{m,n}\geq -a\lambda_1(\Delta_f)-b$ for $0<a\leq\dfrac{m}{m-1}, b\geq0$. In particular, we consider three main cases for different $a$ and $b$ with or without conditions on $\lambda_1(\Delta_f)$. These results are extensions of Dung and Vieira, and weighted generalizations of Li-Wang, Dung-Sung and Vieira.

\end{abstract}

\maketitle


\section{Introduction}
One of the central problem in differential geometry is the relation between the geometry and topology of a manifold. The study of $L^2$-harmonic forms is such kind of problem, see for example \cite{Bue99,Car07,Lott03}, etc., and references there in. Let $M$ be a complete Riemannian manifold and $\lambda_1(M)$ be the first spectrum of the Laplacian on $M$, which can be characterized by
$$
 \lambda_{1}(M)=\inf_{\psi \in C_{0}^{\infty}(M)} \frac{\int_{M}|\nabla \psi|^{2}~\mathrm{d}v}{\int_{M} \psi^{2}~\mathrm{d}v}.
$$

In the work of Li and Wang \cite{LW01}, they proved the following vanishing theorem for $L^2$ harmonic 1-forms on manifolds whose Ricci curvature is bounded below by a negative multiple of the first spectrum,

\begin{thm}[Theorem 4.2, \cite{LW01}]\label{LW01}
 Let $M$ be an $n$-dimensional complete Riemannian manifold with $\lambda_1(M)>0$ and
 $$
  \mathrm{Ric}_{M} \geq-\frac{n}{n-1} \lambda_{1}(M)+\delta,
 $$
 for some $\delta>0$. Then $\mathcal{H}^{1}\left(L^{2}(M)\right)=0$, where $\mathcal{H}^{1}\left(L^{2}(M)\right)$ denotes the space of $L^2$ integrable harmonic 1-forms on $M$.
\end{thm}

After that, Lam \cite{Lam10} generalizes Li-Wang's theorem to manifolds with a weighted Poincar\'e inequality with growth assumption on the weight function. Later, Dung \cite{Du12} asked what is the geometric structure of $M$ if $\delta=0$? Dung and Sung \cite{DS14} get the following theorem,
\begin{thm}[Theorem 2.2, \cite{DS14}]\label{DS14}
 Let $M$ be a complete Riemannian manifold of dimension $n\geq 3$. Suppose that $\lambda_1(M)>0$ and
 $$
  \mathrm{Ric}_{M} \geq-\frac{n}{n-1} \lambda_1(M).
 $$
 Then, either

 $(1)$ $\mathcal{H}^{1}\left(L^{2}(M)\right)=0$; or

 $(2)$ $\widetilde{M}=\mathbb{R} \times N$, where $\widetilde{M}$ is the universal cover of $M$ and $N$ is a manifold of dimension $(n-1)$.
\end{thm}

More generally, Dung \cite{Du12} considered the smooth metric measure space $(M,g,\mathrm{e}^{-f}\mathrm{d}v)$, and denote by $\lambda_1(\Delta_f)$ the first spectrum of the $f$-Laplacian on $M$, which can be similarly characterized by
$$
 \lambda_{1}(\Delta_f)=\inf_{\psi \in C_{0}^{\infty}(M)} \frac{\int_{M}|\nabla \psi|^{2}\cdot\mathrm{e}^{-f}\mathrm{d}v}{\int_{M} \psi^{2}\cdot\mathrm{e}^{-f}\mathrm{d}v},
$$
Then Dung \cite{Du12} proved the following result which is concerned with vanishing for space of $L^2_f$ harmonic functions and splitting of $M$,
\begin{thm}[Theorem 1.3, \cite{Du12}]\label{Du12}
 Let $(M,g,\mathrm{e}^{-f}\mathrm{d}v)$ be a complete non-compact smooth metric measure space of dimension $n\geq 3$ with positive spectrum $\lambda_1(\Delta_f)>0$. Assume that
 $$
  \mathrm{Ric}_{m,n}\geq -\frac{m}{m-1}\lambda_1(\Delta_f).
 $$
 Then either

 $1$. $\mathcal{H}(L_f^2(M))=\mathbb{R}$, where $\mathcal{H}(L_f^2(M))$ is the space of $f$-harmonic functions with finite $f$-energy; or

 $2$. $M=\mathbb{R}\times N$ with the warped product metric
 $$
  \mathrm{d}s_M^2=\mathrm{d}t^2+\eta^2(t)\mathrm{d}s_N^2,
 $$
 where $\eta(t)$ is a positive function and $N$ is an $(n-1)$-dimensional manifold.
\end{thm}

For space of $L_f^2$ harmonic 1-forms, Vieira proved in \cite{Vie13},
\begin{thm}[Theorem 1.1, \cite{Vie13}]\label{Vie1}
 Let $(M,g,\mathrm{e}^{-f}\mathrm{d}v)$ be a complete non-compact smooth metric measure space with non-negative $\infty$-Bakry-\'Emery Ricci curvature. If the space of $L_f^2$ harmonic 1-forms is non-trivial, then the weighted volume of $M^n$ is finite, that is
 $$
 \mathrm{vol}_f(M^n)=\int_{M^n} \mathrm{e}^{-f}\mathrm{d}v<\infty,
 $$
 and the universal covering splits isometrically as $\widetilde{M}^{n}=\mathbb{R} \times N^{n-1}$.
\end{thm}

As a corollary, Vieira obtained (Corollary 1.2, \cite{Vie13}) with the same curvature assumption, if the first eigenvalue of the $f$-Laplacian is positive, then the space of $L_f^2$ harmonic 1-forms is trivial.

Inspired by Li-Wang, Vieira, Dung and Dung-Sung's work, in this paper, we extend Vieira's Theorem \ref{Vie1} by relaxing the curvature condition to be $\mathrm{Ric}_{m,n} \geq-a \lambda_1(\Delta_f)$ without positivity restriction on $\lambda_1(\Delta_f)$, and generalize Dung-Sung's Theorem \ref{DS14} to complete non-compact smooth metric measure spaces, which can also be considered as an extension of Dung's Theorem \ref{Du12}. More precisely, we get
\begin{thm}\label{main1}
 Let $(M,g,\mathrm{e}^{-f}\mathrm{d}v)$ be a complete non-compact smooth metric measure space of dimension $n\geq 3$. If the $m$-Bakry-\'{E}mery Ricci curvature satisfies
 $$
  \mathrm{Ric}_{m,n}(x) \geq-a \lambda_1(\Delta_f),
 $$
 where $0< a < \dfrac{m}{m-1}$, then the space of $L_f^2$ harmonic 1-forms is trivial.
\end{thm}

\begin{thm}\label{main2}
 Let $(M,g,\mathrm{e}^{-f}\mathrm{d}v)$ be a complete non-compact smooth metric measure space of dimension $n\geq 3$ with positive first eigenvalue $\lambda_1(\Delta_f)$. If the $m$-Bakry-\'{E}mery Ricci curvature satisfies
 $$
  \mathrm{Ric}_{m,n} \geq-\frac{m}{m-1} \lambda_1(\Delta_f),
 $$
 then either

 $(1)$. $\mathcal{H}^{1}(L_f^{2}(M))=\{0\}$; or

 $(2)$. $\widetilde{M}=\mathbb{R} \times N$, where $\widetilde{M}$ is the universal cover of $M$ and $N$ is a manifold of dimension $(n-1)$.
\end{thm}

Finally, if $\lambda_1(\Delta_f)$ has some positive lower bound, then conditions on $\mathrm{Ric}_{m,n}$ can be further relaxed (Theorem \ref{m3}), which is in the spirit of Theorem 6 and 7 in \cite{Vie16}.

\section{Preliminaries}
A smooth metric measure space $(M,g,\mathrm{e}^{-f}\mathrm{d}v)$ is a smooth Riemannian manifold $(M,g)$ together with a smooth function $f$ and a measure $\mathrm{e}^{-f}\mathrm{d}v$. For any constant $m\geq n=\dim M$, we have the $m$-dimensional Bakry-\'{E}mery Ricci curvature
$$
 \mathrm{Ric}_{m,n}=\mathrm{Ric}+\nabla^2 f-\frac{\nabla f\otimes \nabla f}{m-n},
$$
which is just the Ricci tensor if $m=n$, and usually we denote by $\displaystyle{\mathrm{Ric}_{f}}=\mathrm{Ric}+\nabla^2 f$ the $\infty$-dimensional Bakry-\'{E}mery Ricci curvature. Hence, in the following proof, we are actually dealing with $m>n$. Analog to $L^2$ differential forms, a differential form $\omega$ is called an $L^2_f$ differential form if
$$
 \int_M |\omega|^2\cdot\mathrm{e}^{-f}\mathrm{d}v<\infty.
$$
By \cite{Bue99}, the formal adjoint of the exterior derivative $\mathrm{d}$ with respect to the $L^2_{f}$ inner product is
$$
 \delta_f=\delta+\iota_{\nabla f}.
$$
Then the $f$-Hodge Laplacian operator is defined as
$$
 \Delta_f=-(\mathrm{d}\delta_f+\delta_f\mathrm{d}).
$$

Since the first spectrum of the weighted Laplacian $\Delta_f$ is given by
$$\lambda_1(\Delta_f)=\inf _{\psi \in C_{0}^{\infty}(M)} \frac{\int_{M}|\nabla \psi|^{2} \cdot\mathrm{e}^{-f}\mathrm{d}v}{\int_{M} \psi^{2} \cdot\mathrm{e}^{-f}\mathrm{d}v},$$
by variational principle, we have the following Poincar\'e type inequality,
$$
 \lambda_1(\Delta_f)\int_{M} \psi^{2}\cdot\mathrm{e}^{-f}\mathrm{d}v \leq \int_{M}|\nabla \psi|^{2}\cdot\mathrm{e}^{-f}\mathrm{d}v \quad \text { for } \psi \in C_{0}^{\infty}(M).
$$

Next, we establish and recall some Lemmas to be used latter. By a smart application of an elementary inequality
\begin{eqnarray}\label{elem inequ}
 (a+b)^2\leq\frac{a^2}{1+\alpha}-\frac{b^2}{\alpha},~\forall\alpha>0,
\end{eqnarray}
Li obtain a Bochner type inequality for $f$-harmonic functions (Lemma 2.1 in \cite{Li05}). Here we adopt Li's idea \cite{Li05} to get the following Bochner type inequality for $L_f^2$ harmonic 1-forms, which will play a key role in this paper,
\begin{lem}
 Let $\omega$ be an $L_f^2$ harmonic 1-form on an $n$-dimensional complete smooth metric measure space $(M,g,\mathrm{e}^{-f}\mathrm{d}v)$ and $m\geq n$ be any constant. Then
 \begin{eqnarray}\label{bochner inequality}
  |\omega|\Delta_f|\omega|\geq\frac{|\nabla|\omega||^2}{m-1}+\mathrm{Ric}_{m,n}(\omega,\omega).
 \end{eqnarray}
 Equality holds iff
 $$
  \left(\omega_{i,j}\right)=\left(\begin{array}{ccccc}{-(m-1) \mu} & {0} & {0} & {\ldots} & {0} \\
  {0} & {\mu} & {0} & {\ldots} & {0} \\ {0} & {0} & {\mu} & {\ldots} & {0} \\
  {\vdots} &  {\vdots} & {\vdots} & {\ddots} & {\vdots} \\
  {0} &  {0} & {0} & {\ldots} & {\mu}\end{array}\right),
 $$
 where $\mu=\dfrac{\langle\nabla f,\omega\rangle}{n-m}$.
\end{lem}

\begin{proof}
 By the weighted Bochner's formula (Equation 2.10 in \cite{Lott03} and Lemma 3.1 in \cite{Vie13}),
 \begin{eqnarray*}
  \frac{1}{2}\Delta_f |\omega|^2
  = |\nabla \omega|^2+\langle \Delta_{f}\omega,\omega \rangle+\mathrm{Ric}_{m,n}(\omega,\omega)+\frac{\nabla f\otimes\nabla f}{m-n}(\omega,\omega).
 \end{eqnarray*}

 It also holds
 $
  \frac{1}{2}\Delta_f|\omega|^2=|\omega|\Delta_f |\omega|+|\nabla |\omega||^2.
 $
 Hence, if $\Delta_f \omega=0$, which by Lemma 2.2 in \cite{Vie13}, is equivalent to that
 \begin{eqnarray}
  \left\{
  \begin{array}{ccc}
    \displaystyle{\omega_{i,j}=\omega_{j,i},~~i,j=1,\cdots,n}\\
    \displaystyle{\sum_{i=1}^n \omega_{i,i}=\langle \nabla f,\omega\rangle},
  \end{array}\label{harmonic form}
  \right.
 \end{eqnarray}
 we have
 $$
  |\omega|\Delta_f|\omega|=|\nabla \omega|^2-|\nabla|\omega||^2+\mathrm{Ric}_{m,n}(\omega,\omega)+\frac{\nabla f\otimes\nabla f}{m-n}(\omega,\omega).
 $$

 Choose an appropriate local frame such that $\displaystyle{\omega_1=\frac{\omega}{|\omega|}}$, then
 $
  |\nabla|\omega||^2=\sum_{j=1}^n \omega_{1,j}^2.
 $
 Since
 \begin{eqnarray}\label{equ1}
  |\nabla\omega|^2 &=& \omega_{1,1}+\sum_{j=2}^n \omega_{1,j}^2+\sum_{j=2}\omega_{j,1}^2+\sum_{i=2}^n\omega_{i,i}^2+\sum_{i,j=2,i\not=j}^n \omega_{ij}^2 \nonumber\\
  &\geq& \omega_{1,1}^2+2\sum_{j=2}^n \omega_{1,j}^2+\frac{1}{n-1}(\sum_{i=2}^n\omega_{i,i})^2\\
  &=& \omega_{1,1}^2+2\sum_{j=2}^n \omega_{1,j}^2+\frac{1}{n-1}(-\omega_{1,1}+\langle\nabla f,\omega\rangle)^2, \nonumber
 \end{eqnarray}
 by (\ref{elem inequ}), we have for any positive number $\alpha$,
 \begin{eqnarray}
  |\nabla\omega|^2-|\nabla|\omega||^2 &=& \sum_{j=2}^n \omega_{1,j}^2+\frac{1}{n-1}(-\omega_{1,1}+\langle\nabla f,\omega\rangle)^2 \nonumber\\
  &\geq& \sum_{j=2}^n \omega_{1,j}^2+\frac{1}{n-1}(\frac{\omega_{1,1}^2}{1+\alpha}-\frac{\langle\nabla f,\omega\rangle^2}{\alpha}) \label{equ2}\\
  &\geq&  \frac{1}{(1+\alpha)(n-1)}\sum_{j=1}^n\omega_{1,j}^2-\frac{1}{\alpha(n-1)}\langle\nabla f,\omega\rangle^2.\label{equ3}
 \end{eqnarray}
 Hence, if we choose $\displaystyle{\alpha=\frac{m-n}{n-1}}$,
 \begin{eqnarray*}
  && |\omega|\Delta_f|\omega|\\
  &\geq& \frac{1}{(1+\alpha)(n-1)}\sum_{j=1}^n\omega_{1,j}^2+\mathrm{Ric}_{m,n}(\omega,\omega)+\frac{\langle\nabla f,\omega\rangle^2}{m-n}-\frac{\langle\nabla f,\omega\rangle^2}{\alpha(n-1)}\\
  &=& \frac{|\nabla|\omega||^2}{m-1}+\mathrm{Ric}_{m,n}(\omega,\omega).
 \end{eqnarray*}
 In addition, equality in (\ref{bochner inequality}) holds if and only if equality in (\ref{equ1}),(\ref{equ2}),(\ref{equ3}) holds simutanously. Then "=" in (\ref{equ1}) implies $\omega_{i,j}=0$ for $2\leq i\not=j\leq n$ and $\omega_{2,2}=\omega_{3,3}=\cdots=\omega_{n,n}$; "=" in (\ref{equ2}) implies $\omega_{1,1}=\dfrac{m-1}{m-n}\langle \nabla f,\omega\rangle$; "=" in (\ref{equ3}) implies $\omega_{1,j}=0$ for $j=2,\cdots,n$. Hence, by (\ref{harmonic form}) and let $\mu=\omega_{2,2}=\dfrac{\langle\nabla f,\omega\rangle}{n-m}$, we finish the proof.
\end{proof}

In the proofs of the sequel, we will always use the following cut-off function,
  \begin{eqnarray}\label{cutoff}
  \phi=\left\{
  \begin{array}{ccc}
    &1,&~~~~\mathrm{on}~B(R),\\
    &0,&~~~~\mathrm{on}~M\backslash B(2R),
  \end{array}
  \right.
 \end{eqnarray}
 such that $\displaystyle{|\nabla\phi|^2\leq \frac{C}{R^2}}$ on $B(2R) \backslash B(R)$.

 The following lemma is a weighted version of the corresponding lemma in \cite{Lam10}, which can be found in \cite{DS13},
\begin{lem}\label{Di}
 Let $h$ be a nonnegative function satisfying the differential inequality
  $$h\Delta_f h\geq -ah^2+b|\nabla h|^2,$$
 in the weak sense, where $a,b$ are constants and $b>-1$.
 For any $\varepsilon>0$, we have the estimate
  \begin{eqnarray*}
   &&[b(1-\varepsilon)+1] \int_{M}|\nabla(\phi h)|^{2} \cdot\mathrm{e}^{-f}\mathrm{d}v\\
   &\leq& \left(b\left(\frac{1}{\varepsilon}-1\right)+1\right) \int_{M} h^{2}|\nabla \phi|^{2} \cdot\mathrm{e}^{-f}\mathrm{d}v+ a \int_{M} \phi^{2} h^{2} \cdot\mathrm{e}^{-f}\mathrm{d}v,
  \end{eqnarray*}
 for any compactly supported smooth function $\phi\in C_0^{\infty}(M)$. In addition, if
  $$
   \int_{B_p(R)} h^2 \cdot\mathrm{e}^{-f}\mathrm{d}v=o(R^2),
  $$
 then
  $$
   \int_{M} |\nabla h|^2 \cdot\mathrm{e}^{-f}\mathrm{d}v \leq \frac{a}{b+1}\int_M h^2\cdot\mathrm{e}^{-f}\mathrm{d}v.
  $$
 In particular, $h$ has a finite $f$-Dirichlet integral if $h\in L^2_f(M)$.
\end{lem}

We will also need the following result,
\begin{lem}\label{convergence}
 For an $L^2_f$ integrable function $h$ on $(M,g,\mathrm{e}^{-f}\mathrm{d}v)$ satisfying the differential inequality
  $$h\Delta_f h\geq -ah^2+b|\nabla h|^2,$$ we have
 \begin{eqnarray}
 &&\lim_{R\rightarrow\infty} \int_M \phi h\langle \nabla\phi,\nabla h \rangle\cdot\mathrm{e}^{-f}\mathrm{d}v=0,\label{int lim2}\\
 &&\lim_{R\rightarrow\infty} \int_M |\nabla(\phi h)|^2\cdot\mathrm{e}^{-f}\mathrm{d}v=\int_M  |\nabla h|^2\cdot\mathrm{e}^{-f}\mathrm{d}v.\label{int lim1}
 \end{eqnarray}
 Moreover,
 \begin{eqnarray}
 && \lambda_1(\Delta_f)\int_M  h^2\cdot\mathrm{e}^{-f}\mathrm{d}v\leq \int_M  |\nabla h|^2\cdot\mathrm{e}^{-f}\mathrm{d}v.\label{int lim3}
 \end{eqnarray}
\end{lem}

\begin{proof}
 By Lemma \ref{Di}, we have
 $$\int_M |\nabla h |^2\cdot\mathrm{e}^{-f}\mathrm{d}v<\infty.$$
 In addition,
 \begin{eqnarray*}
  && \int_M |\nabla(\phi h)|^2\cdot\mathrm{e}^{-f}\mathrm{d}v\\
  &=& \int_M h^2|\nabla\phi |^2\cdot\mathrm{e}^{-f}\mathrm{d}v + \int_M \phi^2|\nabla h |^2\cdot\mathrm{e}^{-f}\mathrm{d}v + 2\int_M \langle h\nabla\phi,\phi\nabla h\rangle\cdot\mathrm{e}^{-f}\mathrm{d}v,
 \end{eqnarray*}
 where
 \begin{eqnarray*}
  && \left|\int_M h^2|\nabla\phi |^2\cdot\mathrm{e}^{-f}\mathrm{d}v + 2\int_M \langle h\nabla\phi,\phi\nabla h\rangle\cdot\mathrm{e}^{-f}\mathrm{d}v\right|\\
  &\leq& \frac{C}{R^2}\int_M h^2\cdot\mathrm{e}^{-f}\mathrm{d}v
  + 2\left(\int_M h^2|\nabla \phi|^2\cdot\mathrm{e}^{-f}\mathrm{d}v\right)^{\frac{1}{2}}\left(\int_M \phi^2|\nabla h |^2\cdot\mathrm{e}^{-f}\mathrm{d}v\right)^{\frac{1}{2}}.
 \end{eqnarray*}
 Hence, letting $R\rightarrow \infty$, one gets
 $$
  \lim_{R\rightarrow\infty} \int_M |\nabla(\phi h)|^2\cdot\mathrm{e}^{-f}\mathrm{d}v=\int_M  |\nabla h|^2\cdot\mathrm{e}^{-f}\mathrm{d}v.
 $$
 From the proof, we see (\ref{int lim2}) holds.
 By variational principle,
 $$
  \lambda_1(\Delta_f)\int_M  (\phi h)^2\cdot\mathrm{e}^{-f}\mathrm{d}v\leq \int_M  |\nabla(\phi h)|^2\cdot\mathrm{e}^{-f}\mathrm{d}v,
 $$
 and let $R\rightarrow\infty$, we obtain (\ref{int lim3}).
\end{proof}

\section{Vanishing and splitting results}

\begin{proof}(proof of Theorem \ref{main1})
 We adopt Vieira's idea in \cite{Vie16} for the proof. Suppose $\mathcal{H}^1(L_f^2)$ is non-trivial, and for any non-trivial $L_f^2$ harmonic 1-form $\omega$, let $h=|\omega|$. Then by (\ref{bochner inequality}) and $\mathrm{Ric}_{m,n}\geq -a\lambda_1(\Delta_f)$, we have
 \begin{eqnarray}\label{bochner func}
  h\Delta_f h\geq\frac{|\nabla h|^2}{m-1}-a\lambda_1(\Delta_f)h^2.
 \end{eqnarray}
 Multiple the cut-off function $\phi^2$ on both sides of (\ref{bochner func}) and by integration by parts, we get
 \begin{eqnarray*}
  &&\frac{1}{m-1}\int_M \phi^2|\nabla h|^2-a\lambda_1(\Delta_f)\int_M \phi^2h^2\\
  &\leq& -\int_M \phi^2|\nabla h|^2 - 2\int_M \phi h\langle \nabla\phi,\nabla h\rangle,
 \end{eqnarray*}
 so
 \begin{eqnarray*}
  &&\frac{m}{m-1}\int_M \phi^2|\nabla h|^2\\
  &\leq& a\lambda_1(\Delta_f)\int_M \phi^2h^2 - 2\int_M \phi h\langle \nabla\phi,\nabla h\rangle\\
  &\leq& a\int_M |\nabla(\phi h)|^2-2\int_M \phi h\langle \nabla\phi,\nabla h\rangle.
 \end{eqnarray*}
 Hence, by Lemma \ref{convergence}, letting $R\rightarrow\infty$, we obtain
 $$
  (\frac{m}{m-1}-a)\int_M |\nabla h|^2\leq 0.
 $$
 Since $a<\frac{m}{m-1}$, $h$ must be a constant. Then the weighted volume
 $$\mathrm{vol}_f(M)=\frac{\int_M h^2}{h^2}<\infty.$$

 If $\lambda_1(\Delta_f)>0$,
 $$
 \lambda_1(\Delta_f)\int \phi^2\leq \int_M |\nabla\phi|^2\leq \frac{C}{R^2}\mathrm{vol}_f(M)\rightarrow 0,~\mathrm{as}~R\rightarrow\infty.
 $$
 This forces $\phi\equiv 0$, which contradicts with the choice of $\phi$. Hence, $\lambda_1(\Delta_f)=0$. Then the curvature condition becomes $\mathrm{Ric}_{m,n}\geq 0$, and by Theorem 9.1 in \cite{Li05}, which says that the weighted volume of a complete non-compact smooth metric measure space with non-negative $\mathrm{Ric}_{m,n}$ is infinite, we get a contradiction. Hence, the space $\mathcal{H}^1(L_f^2(M))$ has to be trivial.
\end{proof}

From the proof, we see that by using similar arguments, Theorem 3 in \cite{Vie16}, which says the space of $L^2$ harmonic $1$-forms is trivial on complete $n$-dimensional non-compact Riemannian manifolds satisfying a weighted Poincar\'e inequality with weight function $\rho$ and $\mathrm{Ric}\geq a\rho$ for $0<a<\frac{n}{n-1}$ can also be generalized to metric measure space. 

\begin{proof} (proof of Theorem \ref{main2})
 If $H^1(L^2_f(M))=0$, it's done. Otherwise, let $\omega$ be a non-trivial $L_f$ harmonic 1-form, and let $h=|\omega|$. Then $h$ is $L_f^2$-integrable. By inequality (\ref{bochner inequality}) and the assumption on $\mathrm{Ric}_{m,n}$, we have
 \begin{eqnarray}\label{bih}
  h\Delta_f h\geq -\frac{m\lambda_1(\Delta_f)}{m-1}h^2+\frac{1}{m-1}|\nabla h|^2.
 \end{eqnarray}

By Lemma \ref{convergence},
 \begin{eqnarray}\label{eigen inequ}
  \lambda_1(\Delta_f)\int_M  h^2 \cdot \mathrm{e}^{-f}\mathrm{d}v\leq\int_M  |\nabla h|^2 \cdot \mathrm{e}^{-f}\mathrm{d}v.
 \end{eqnarray}

 Similar to the proof of Theorem 2.2 in \cite{DS14}, if $"<"$ holds in (\ref{eigen inequ}), we get a contradiction. If "=" holds in (\ref{eigen inequ}),
equality of (\ref{bih}) holds, and this forces equality holds in (\ref{bochner inequality}), so by Lemma \ref{bochner inequality},
 $$
  \left(w_{i,j}\right)=\left(\begin{array}{ccccc}{-(m-1) \mu} & {0} & {0} & {\ldots} & {0} \\ {0} & {\mu} & {0} & {\ldots} & {0} \\ {0} & {0} & {\mu} & {\ldots} & {0} \\ {\vdots} & {\vdots} & {\vdots} & {\ddots} & {\vdots} \\ {0} & {0} & {0} & {\ldots} & {\mu}\end{array}\right).
 $$
 The splitting argument is the same as that of Li and Wang \cite{LW06}, page 946, or Dung and Sung \cite{DS14}, page 1788, so we omit here.
\end{proof}

If the $m$-Bakry-\'{E}mery Ricci curvature condition is further relaxed, we will need an extra condition on $\lambda_1(\Delta_f)$,
\begin{thm}\label{m3}
 Let $(M,g,\mathrm{e}^{-f}\mathrm{d}v)$ be a complete non-compact smooth metric measure space of dimension $n \geq 3$. Suppose that $\lambda_1(\Delta_f)\geq \dfrac{b}{\frac{m}{m-1}-a}$ and
  $$\mathrm{Ric}_{m,n}\geq -a\lambda_1(\Delta_f)-b,$$
  where $0<a<\frac{m}{m-1}$ and $b>0$.
 Then, either

 $(1)$. $H^1(L^2_f(M))=\{0\}$; or

 $(2)$. $\widetilde{M}=\mathbb{R}\times N$, where $\widetilde{M}$ is the universal cover of $M$ and $N$ is a manifold of dimension $n-1$.
\end{thm}
\begin{proof}
 For any $L_f^2$-harmonic 1-form $\omega$, let $h=|\omega|$, so we have
 \begin{eqnarray}\label{bochner func2}
  h\Delta_f h\geq \frac{1}{m-1}|\nabla h|^2-a\lambda_1(\Delta_f) h^2-bh^2.
 \end{eqnarray}
 Multiplying the cut-off function $\phi^2$ on both sides of (\ref{bochner func2}) and by integration by parts, one gets
 \begin{eqnarray*}
 && \frac{1}{m-1}\int_M \phi^2|\nabla h|^2\cdot \mathrm{e}^{-f}\mathrm{d}v-a\lambda_1(\Delta_f)\int_M \phi^2 h^2\cdot \mathrm{e}^{-f}\mathrm{d}v-b\int_M \phi^2h^2\cdot \mathrm{e}^{-f}\mathrm{d}v\\
 &\leq& -\int_M \phi^2|\nabla h|^2\cdot \mathrm{e}^{-f}\mathrm{d}v-2\int_M \langle \phi\nabla h,h\nabla \phi\rangle \cdot \mathrm{e}^{-f}\mathrm{d}v.
 \end{eqnarray*}
 Combining variational principle, one obtains
 \begin{eqnarray*}
  && \frac{m}{m-1}\int_M \phi^2|\nabla h|^2\cdot \mathrm{e}^{-f}\mathrm{d}v\\
  &\leq& a\int_M |\nabla(\phi h)|^2\cdot \mathrm{e}^{-f}\mathrm{d}v+b\int_M \phi^2h^2\cdot \mathrm{e}^{-f}\mathrm{d}v-2\int_M \langle \phi\nabla h,h\nabla \phi\rangle \cdot \mathrm{e}^{-f}\mathrm{d}v.
 \end{eqnarray*}
 Letting $R\rightarrow \infty$, we have
 \begin{eqnarray}\label{int inequal}
  \int_M |\nabla h|^2\cdot \mathrm{e}^{-f}\mathrm{d}v
  \leq \frac{b}{\frac{m}{m-1}-a}\int_M h^2\cdot \mathrm{e}^{-f}\mathrm{d}v.
 \end{eqnarray}
 Suppose $\lambda_1(\Delta_f)>\dfrac{b}{\frac{m}{m-1}-a}$, if $h\not\equiv 0$, then (\ref{int lim3}) and (\ref{int inequal}) implies
 $$\lambda_1(\Delta_f)\leq\frac{b}{\frac{m}{m-1}-a},$$
 which is a contradiction.
 Hence, if $\lambda_1(\Delta_f)>\dfrac{b}{\frac{m}{m-1}-a}$, then $$\mathcal{H}^1(L_f^2(M))=\{ 0 \}.$$

 Suppose $\lambda_1(\Delta_f)=\dfrac{b}{\frac{m}{m-1}-a}$ and $\mathcal{H}^1(L_f^2(M))$ is non-trivial, then equality hold in (\ref{int inequal}). Hence, equality holds in (\ref{bochner func2}), and this forces equality holds in (\ref{bochner inequality}). The rest of the splitting argument is the same as that in the proof of Theorem \ref{main2}.
\end{proof}

From the proof we see that if $\lambda_1(\Delta_f) > \dfrac{b}{\frac{m}{m-1}-a}$, $\mathcal{H}^1(L_f^2(M))$ vanishes, and the splitting case only happens when $\lambda_1(\Delta_f) = \dfrac{b}{\frac{m}{m-1}-a}$.

\vspace{0.5cm}\noindent\textbf{Acknowledgments}.  The author would like to thank Prof. Jiayong Wu for useful suggestions. J.R. Zhou is partially supported by a PRC grant NSFC  11771377 and the Natural Science Foundation of Jiangsu Province(BK20191435).

\end {document}